\documentclass[12pt]{amsart}
\usepackage[latin1]{inputenc}
\usepackage{color}
\usepackage{bbm}
\usepackage{amsmath,amssymb}
\usepackage{enumerate}
\usepackage{mathrsfs}
\usepackage{verbatim}
\usepackage[top=2cm,bottom=2cm,left=3cm,right=3cm]{geometry}

\newtheorem{thm}{Theorem}[section]

\newtheorem{lem}[thm]{Lemma}
\newtheorem{prop}[thm]{Proposition}
\theoremstyle{definition}
\newtheorem{defin}[thm]{Definition}

\newtheorem{rem}[thm]{Remark}

\numberwithin{equation}{section}

\newcommand{\diam}{\operatorname{diam}}

\newcommand{\dif}{\,\mathrm{d}}

\newcommand{\charfun}{\ensuremath{\mathbbm 1}} 
\newcommand{\sm}[1]{\ensuremath{#1'}}  
\newcommand{\la}[1]{\ensuremath{#1''}} 
\newcommand{\pp}{\mathrm{p}} 

\allowdisplaybreaks

\begin{document}
\title
{Variation inequalities for smartingales}

\author[M. Passenbrunner]{Markus Passenbrunner}
\address{Institute of Analysis, Johannes Kepler University Linz, Austria, 4040 Linz, Altenberger Strasse 69}
\email{markus.passenbrunner@jku.at, markus.passenbrunner@gmail.com}

\begin{abstract}
	A result by N.G. Makarov [Algebra i Analiz, 1989] states that for dyadic martingales 
	$(M_n)$ on the torus we have the strict inequality 
	\[
		\liminf_{n\to\infty} \frac{M_n}{\sum_{k=1}^n |\Delta M_k|}	 > 0
	\]
	on a set of Hausdorff dimension one, denoting by $\Delta M_n$ the 
	martingale differences $ \Delta M_n = M_n - M_{n-1} $. We discuss an extension of this 
	inequality to so-called smartingales on convex, compact subsets of $\mathbb R^d$, which  
	are piecewise polynomial (or \emph{s}pline) versions of martingales.  
	As a tool we need and prove an estimate for smartingales in the spirit of 
	the law of the iterated logarithm.
\end{abstract}
\maketitle 

\section{Introduction and Preliminaries}

Let $(M_n,\mathscr F_n)_{n\geq 0}$ be a martingale defined on a probability space $(\Omega,\mathscr F,\mu)$
with $M_0=0$ and $\mathscr F_0 = \{\emptyset,\Omega\}$ and let $\Delta M_n = M_n - M_{n-1}$ for $n\geq 1$.
Define the square function
\[
S_n^2 = \sum_{i=1}^n \mathbb E_{i-1}(\Delta M_i)^2 
\]
and $S_\infty^2 = \sup_n S_n^2 = \lim_n S_n^2$,
denoting by $\mathbb E_n$ the conditional expectation 
with respect to $\mathscr F_n$.
By a result of W. Stout \cite{Stout1970}, martingales enjoy the law of 
the iterated logarithm (LIL). 
In particular, assuming that the differences $\Delta M_n$ of $(M_n)$ are uniformly bounded  
(i.e., there is a constant $L$ such that for all $n$ we have $|\Delta M_n|\leq L$),
we have  $\mu$-a.s. on $\{ S_\infty = \infty\}$
\begin{align*}
	\limsup_{n\to \infty} \frac{M_n}{\sqrt{ 2 S_n^2 \log\log S_n^2}} = 1,\qquad 
\liminf_{n\to \infty} \frac{M_n}{\sqrt{ 2 S_n^2 \log\log S_n^2}} = -1. 
\end{align*}
On the set $\{ S_\infty < \infty\}$, the martingale $(M_n)$ has a finite limit $\mu$-almost surely.
For more information on the LIL and its broad use in analysis, see the monograph \cite{BanuelosMoore1999}
by R.~Ba\~nuelos and C.N.~Moore.

N.G. Makarov \cite[Proposition 6.16]{Makarov1989} showed via change of measure arguments that 
this law of iterated logarithm for martingales implies for dyadic
martingales $(M_n)$  on the torus with uniformly bounded differences the inequality 
\begin{equation}\label{eq:ineq_makarov}
\liminf_{n\to\infty} \frac{M_n}{\sum_{k=1}^n |\Delta M_k|}	 > 0
\end{equation}
on a set of Hausdorff dimension one. 
Thus, there are many points where the martingale $(M_n)$ is comparable to its variation.

We mention here similar results for the radial variation of analytic functions on 
the unit disc. J.~Bourgain \cite{Bourgain1993} showed that for bounded analytic (or harmonic) functions $F$ on the 
unit disc, 
the radial variation
\[
	V(\theta) = \int_0^1 |F'(\rho e^{i\theta})| \dif \rho
\]
is bounded on a set of Hausdorff dimension $1$.
More generally, let $b$ be a  
Bloch function, i.e.,
an analytic function $b$ on the unit disc $\mathbb D$ satisfying 
\[ 
	\sup_{z\in \mathbb D} (1-|z|^2) |b'(z)| < \infty.
\]
Then, based on Bourgain's result,  P.W.~Jones and P.F.X.~M\"uller \cite{JonesMueller1997} showed that  
there exists 
an angle $\theta\in [0,2\pi)$ satisfying 
\[
\liminf_{r\to 1} \frac{\Re b(r e^{i\theta})}{\int_0^r |b'(\rho e^{i\theta})|\dif \rho} > 0,
\]
where $\Re b$ denotes the real part of $b$. This can be seen as
 an analogue of inequality \eqref{eq:ineq_makarov} for Bloch functions and 
 it solved a conjecture by J.M.~Anderson \cite{Anderson1971} that for any 
 univalent function $f$ on the unit disc,
  there exists a point $\theta\in [0,2\pi)$ such that 
  \begin{equation}\label{eq:anderson}
 \int_0^1 |f''(r e^{i\theta})|\dif r < \infty.
  \end{equation}
 This is done by associating to 
 $f$ the Bloch function $b = \log f'$. We also mention here the extension 
 of \eqref{eq:anderson} from dimension $2$ to arbitrary dimension by 
 P.F.X.~M\"uller and K.~Riegler \cite{MuellerRiegler2020}.
 For more information on the boundary behavior of analytic functions, we refer
 to the book \cite{Pommerenke1992} by Ch. Pommerenke.

In this article we are concerned with orthonormal systems that have similar 
support properties as martingale differences and consist of piecewise polynomials 
or splines instead of piecewise constant 
functions and we  are going to show inequalities in the spirit of \eqref{eq:ineq_makarov} for such systems.
In \cite{part1, part2, part3}, such orthonormal systems were investigated 
in a quite broad sense that we now describe.

Let $(\mathscr F_n)_{n\geq 0}$ be  a filtration on the set $\Omega$ 
such that $\mathscr F_0 = \{\emptyset,\Omega\}$ and for each $n\geq 1$, 
$\mathscr F_n$ is generated by the subdivision of each atom $A$ of $\mathscr F_{n-1}$
into two atoms 
$\sm{A},\la{A}$ of $\mathscr F_{n}$ satisfying $\mu(\la{A}) \geq \mu(\sm{A})>0$. 
Moreover, denote the collection of all atoms in $\mathscr F_n$ by $\mathscr A_n$ and define $\mathscr A
 = \cup_n \mathscr A_n$.
We also let 
$S\subseteq L^\infty$ be a vector space of real valued $\mathscr F$-measurable functions on $\Omega$
of finite dimension.
Denoting  $\|f\|_A :=\|f\|_{L^\infty(A)}$,
we assume  for each $A\in\mathscr A$ the inequality
	\begin{equation}\label{eq:L1Linfty}
			\mu(\{ \omega\in A : |f(\omega)| \geq c_1\|f\|_{A}
		\}) \geq c_2\mu(A),\qquad f\in S
	\end{equation}
for some constants $c_1,c_2\in (0,1]$ independent of $A$.
It can be checked that inequality \eqref{eq:L1Linfty} for fixed $A$ is 
equivalent to the comparability of $\|f\|_A$ and the mean value
 $\mu(A)^{-1}\int_A|f|\dif\mu$ for all $f\in S$.
For each $n\geq 0$ set
\[
S_n := \{ f:\Omega\to\mathbb R\; |\; \text{for each atom of $\mathscr F_n$ 
	there exists $g\in S$ so that $f\charfun_A = g\charfun_A$}\}.
\]
Thus, $S_n$ consists of those functions that are piecewise (on each atom of $\mathscr F_n$)
 contained in $S$.
Let $P_n$ be the orthoprojector onto $S_n$ with respect to the inner product 
on the space $L^2(\Omega,\mathscr F,\mu)$.
Additionally, we denote $S(A)= \{f\cdot \charfun_A : f\in S\}$ for $A\subset \Omega$.

In the articles \cite{part1,part2, part3} we investigate properties of orthonormal 
systems resulting from this pretty general setting of data. Examples of  results that 
hold true already under the above assumptions is the a.e. convergence of $P_n f$ for 
$f\in L^1$
or that the series of differences $(P_n - P_{n-1}) f$
converges unconditionally in $L^p$ for $p\in (1,\infty)$, i.e. we have the analogue 
of Burkholder's inequality for martingale differences.
The main class of examples considered in the papers 
\cite{part1,part2, part3}, that satisfy inequality \eqref{eq:L1Linfty},
are polynomial spaces $S$.
We note that the framework above only allows piecewise polynomials without 
smoothness conditions, but the aforementioned results about convergence 
of spline difference series are also true in the 
setting of smooth splines. For many results in this direction, we 
refer to \cite{m4,m5,m9,m11,m14,m13,m17,m15,m16,m20}.

In contrast to those convergence results, the underlying article analyzes the 
asymptotic behavior of unbounded spline 
difference series from the viewpoint of inequality \eqref{eq:ineq_makarov} with
the martingale $(M_n)$ replaced by its spline version, i.e. a sequence $(f_n)$
satisfying $P_{n-1} f_n = f_{n-1}$.

In what follows we use the notation $A(t)\lesssim B(t)$ if there exists a constant $C$ that depends 
only on the data $(\Omega,\mathscr F,\mu)$ and $S$ satisfying $A(t)\leq C B(t)$, 
where $t$ denotes all implicit and explicit dependencies that the expressions $A$ and 
$B$ might have. Similarly we use the symbols $\gtrsim$ and $\simeq$.

From now on we only consider data 
$(\Omega,\mathscr F,\mu), S, \mathscr A$ satisfying the following 
(additional) assumptions:
\begin{enumerate}
	\item The probability space $(\Omega,\mathscr F,\mu)$ consists of some convex compact subset 
	$\Omega$ of $\mathbb R^d$, the Borel $\sigma$-algebra $\mathscr F$ on $\Omega$, and 
	some probability measure $\mu$ on $\Omega$.
	\item The finite-dimensional vector space $S$ consists of polynomials on $\Omega$ such 
		that there exists a function $g\in S$ satisfying $g : \Omega \to [c_3,1]$
		for some positive constant $c_3$.
	\item The atoms $\mathscr A$ consist of convex subsets $A$ of $\Omega$ that, 
			in combination with the measure $\mu$, satisfy inquality \eqref{eq:L1Linfty}, 
			and such that 
			\begin{equation} \label{eq:comp_w_diam}
				  w(A) \simeq \diam A := \sup_{x,y\in A} |x-y|, 
			\end{equation}
			denoting by $w(A)$ the minimal distance between two parallel supporting hyperplanes of $A$.
		Additionally, assume that  $\max_{A\in\mathscr A_n} \diam A \to 0$ as $n\to\infty$.
\end{enumerate}

We give a few remarks regarding those conditions. Note that if the space of polynomials 
$S$ is such that the constant function is contained in $S$, we can choose $g \equiv 1$ in
condition (2). Moreover, in the case of polynomials on the real line ($d=1$), the 
comparability condition \eqref{eq:comp_w_diam} is no restriction whereas for $d>1$, inequality
 \eqref{eq:comp_w_diam} states that  
atoms are approximately of the same size in every direction.

We now discuss more properties that are satisfied by spaces of polynomials $S$, in particular 
regarding Lebesgue measure. 
First we note that inequality \eqref{eq:L1Linfty} is a consequence of 
Remez' inequality (see for instance Ju.A. Brudnyi, M.I. Ganzburg \cite{BrudnyiGanzburg1973} or 
M.I. Ganzburg \cite{Ga2001}) stating that there exist constants $C,n$
such that  for each convex set $A$, each measurable subset $E\subset A$ and each function $f\in S$,
\begin{equation}\label{eq:remez}
\| f\|_A \leq  \Big( \frac{C|A|}{|E|} \Big)^n \|f\|_E,
\end{equation}
where we denote (here and in what follows) by $\mathbb P =|\cdot|$ the 
$d$-dimensional Lebesgue measure normalized so that $|\Omega| = 1$. 
We also have Markov's inequality for the boundedness of 
Lipschitz constants of polynomials, which states that 
there exists a constant $C$ such that for 
each atom $A$ and each function $f\in S$ we have the inequality 
\begin{equation}\label{eq:lip_condition}
\frac{|f(x) - f(y)|}{|x-y|} \leq C \frac{\|f\|_A}{\diam A},\qquad x,y\in A, x\neq y,
\end{equation}
(see D.A. Wilhelmsen \cite{Wilhelmsen1974} or A. Kro\'{o}, S. R\'{e}v\'{e}sz \cite{KrooRevesz1999}
for $d>1$), which takes this form by the comparability assumption \eqref{eq:comp_w_diam} 
 between $w(A)$ and $\diam A$.

Let now $(f_n)$ be a sequence of integrable functions on $(\Omega,\mathscr F,\mu)$, 
which, together with the space $S$ and the atoms $\mathscr A$, satisfies conditions 
(1)--(3) above. 
In order for it to be quickly distinguishable from martingales, 
we say that $(f_n)$ is 
a \emph{smartingale} if it satisfies 
\[
	P_{n-1} f_n = f_{n-1}, \qquad n\geq 1.
\]
The functions $f_n$ will then be piecewise polynomials or \emph{s}plines,
the first letter of this word being the reason for the name \emph{smartingales}.

In this article we show a result in the spirit of \eqref{eq:ineq_makarov} for 
smartingales, which is contained in Theorem~\ref{thm:application}.
In Sections~\ref{sec:tilde_measure} and \ref{sec:lil_smart} we develop the tools 
used in the proof of the aforementioned theorem, which are 
constructions of new measures corresponding to smartingales 
 and part of a law of iterated logarithm for smartingales, respectively.

\section{Change of measure}\label{sec:tilde_measure}
In this section, we construct measures $\tilde{\mathbb P}$ with respect 
to which small perturbations of smartingales are again smartingales.
Let $(f_n)$ be a smartingale with respect to the filtration $(\mathscr F_n)$ and 
the  $d$-dimensional Lebesgue measure $\mathbb P = |\cdot |$.
As for martingales, we use the notation $\Delta f_n = f_n - f_{n-1}$.
We fix the parameter $\lambda >0$ sufficiently small.
Define the new sequence of functions $(\tilde{f}_n)$ via their differences 
$\Delta \tilde{f}_n$ by 
\begin{equation}\label{eq:def_f_tilde}
\Delta \tilde{f}_n = \Delta f_n - \lambda p_n,
\end{equation}
for some non-negative $p_n\in S_n$ such that for each $x\in\Omega$ we have  
$\mathbb E_n p_n(x) \simeq \mathbb E_n |\Delta f_n|(x)$ with some absolute implicit constants.

Note that in the martingale case, with the choice $p_n = |\Delta f_n|$,
 this definition of $\tilde{f}_n$
equals $\tilde{f}_n = f_n - \lambda \sum_{\ell=1}^n |\Delta f_\ell|$, 
considered in N.G. Makarov \cite{Makarov1989}.
We will construct a measure $\tilde{\mathbb P}$ such that $(\tilde{f}_n)$ is 
a smartingale with respect to $\tilde{\mathbb P}$, i.e.
with the projector $\tilde{P}_n$ onto $S_n$ with respect to the measure $\tilde{\mathbb P}$,
we have 
\[
	\tilde{P}_{n-1} \tilde{f}_n = \tilde{f}_{n-1},\qquad n\in\mathbb N.
\]
The measure  $\tilde{\mathbb P}$ will be the limit of some 
measures $\tilde{\mathbb P}_n$ of the form $\dif\tilde{\mathbb P}_n = d_n \dif\mathbb P$
for some $\mathscr F_n$-measurable densities $d_n$. 
Starting with $\tilde{\mathbb P}_0 = \mathbb P$, 
the measures $\tilde{\mathbb P}_n$ have to then satisfy the smartingale condition 
\begin{equation}\label{eq:cond}
0 = \int_A \Delta\tilde{f}_\ell \cdot f \dif\tilde{\mathbb P}_n 
= \int_{A} (\Delta f_\ell - \lambda p_\ell) \cdot f \dif\tilde{\mathbb P}_n 
 \qquad f\in S, A\in\mathscr A_{n-1}
\end{equation}
for the differences $\Delta \tilde{f}_\ell$ for each $\ell\leq n$.
When giving an iterative formula for the measures $\tilde{\mathbb P}_n$, it is enough 
to have the smartingale condition \eqref{eq:cond} only for $\ell = n$ if we additionally assume 
the compatibility conditions 
\begin{equation}\label{eq:compatibility}
\int_{A} q \dif\tilde{\mathbb P}_{n} = \int_{A} q \dif\tilde{\mathbb P}_{n-1}, 
\qquad q\in S^2 := \operatorname{span}\{ g_1\cdot g_2 : g_1,g_2\in S\}.
\end{equation}

We will perform the inductive construction of the sequence $\tilde{\mathbb P}_n$ and 
$\tilde{\mathbb P}$ in the two cases 
$\dim S =1 $ and $\dim S >1$ separately.
The motivation for presenting the simple case of $\dim S = 1$ is twofold. Firstly, 
we don't need 
to assume additional constraints on the atoms and secondly
we also don't have to pass to sparse smartingales (see Definition~\ref{def:sparse}) 
as in the case of spaces $S$ with $\dim S > 1$
presented in Section~\ref{sec:higher_dim_S}.

The result in both cases is that the desired measures $\tilde{\mathbb P}$ can 
indeed be constructed (see Theorem~\ref{thm:dim1} and Theorem~\ref{thm:tilde_P}).
Moreover, we give additional estimates of the measure $\tilde{\mathbb P}$ in
terms of the Lebesgue measure $\mathbb P$ that help us in proving that the set of points where the inequality
\[
	\liminf_{n\to\infty} \frac{f_n}{\sum_{k\leq n} \mathbb E_k |\Delta f_k|} >0
\]
is satisfied, has full Hausdorff dimension (see Theorem~\ref{thm:application} in the end).
In the case of $\dim S >1$, showing the mere existence of $\tilde{\mathbb P}$ is not 
as easy as in the martingale or $\dim S =1$ case. 
To guarantee the solvability 
of the involved linear systems for the measures $\tilde{\mathbb P}_n$ (see~\eqref{eq:cond} and \eqref{eq:compatibility}),
 we need a few auxiliary results that are 
given in Lemmas~\ref{lem:product}, \ref{lem:existence_ell},
~\ref{lem:independent}, and~\ref{lem:det_est}.

\subsection{The case $\dim S= 1$}
We first investigate the case $\dim S=1$ and
 let $S = \operatorname{span}\{ g\}$
for some polynomial $g:\Omega\to [c_3,1]$ with
some positive constant $c_3$. 

We set $\nu(A) = \int_A g^2 \dif\mathbb P$, 
which is a finite measure on $\Omega$.
Assume that 
$A\in\mathscr A_{n-1}$ is an atom of $\mathscr F_{n-1}$.
Then, we have 
\[
	\Delta f_n = \alpha ( \nu(A'') g \charfun_{A'} - \nu(A') g\charfun_{A''}),\qquad\text{on }A,
\]
for some $\alpha\in\mathbb R$, since $\int_{A} g\cdot \Delta f_n  \dif\mathbb P = 0$. 
We assume the normalization $\alpha = \varepsilon$ for some sign $\varepsilon\in \{\pm 1\}$.
Since $0<c_3\leq g\leq 1$, we can choose the functions $p_n$ in~\eqref{eq:def_f_tilde} such that 
\begin{equation}\label{eq:def_tilde_1D}
\Delta\tilde{f}_n = \Delta f_n - \lambda\big(  \nu(A'') g \charfun_{A'} + \nu(A') g \charfun_{A''} \big).
\end{equation}
With the starting function $d_0 = 1$, we will construct a sequence of densities 
$d_n \geq 0$ (with respect to the measure $\mathbb P$) such that  on $A$ we have 
$d_n = d_n'\charfun_{A'} + d_n''\charfun_{A''}$ for some (positive) numbers $d_n',d_n''$ and we 
set $\dif \tilde{\mathbb P}_n = d_n \dif \mathbb P$.
Conditions \eqref{eq:cond} and \eqref{eq:compatibility} translate
to the following conditions for the numbers $d_n'$ and $d_n''$:
\begin{align*}
	\nu(A')\nu(A'') \big(d_n' (\varepsilon  -\lambda) +d_n'' ( - \varepsilon - \lambda) \big)  &= 0,\\
	d_n' \nu(A') + d_n'' \nu(A'') &= d_{n-1}(A)\nu(A),
\end{align*}
denoting by $d_{n-1}(A)$ the value of the $\mathscr F_{n-1}$-measurable function 
$d_{n-1}$ on the atom $A\in\mathscr A_{n-1}$.
Corresponding to this system of equations,
we have the matrix 
\[
B_\lambda = \begin{pmatrix}
	\varepsilon-\lambda &  -\varepsilon -\lambda \\
	\mu & 1-\mu
\end{pmatrix}
\]
with the number $\mu = \nu(A')/\nu(A)$ and  the right hand side $(0,d_{n-1}(A))^T$,
Then, $\det B_\lambda = \varepsilon + \lambda(2\mu-1)$ and therefore 
\begin{equation}
	\label{eq:dnprime}
\begin{pmatrix}
	d_n' \\
	d_n''
\end{pmatrix}
= \frac{d_{n-1}(A)}{\varepsilon+\lambda(2\mu-1)} 
\begin{pmatrix}
	\varepsilon + \lambda \\
	\varepsilon - \lambda
\end{pmatrix}.
\end{equation}

We next want to give some 
estimates for $\mathbb P_n(B) = d_n(B) \mathbb P(B)$ for atoms $B\in\mathscr A_n$.
This is the purpose of the following couple of lemmas.

\begin{lem}\label{lem:triv}
	For every $\lambda\in (0,1)$ and every $\eta\in (-1,1)$,
	\[
		(1+\eta) \leq \min \big(2 (1+\eta\lambda), 3(1-\eta\lambda) \big)
	\]
\end{lem}
\begin{proof}
The first inequality is equivalent to the inequality $\eta(1-2\lambda) \leq 1$, which is clearly true for 
those ranges of $\lambda,\eta$. The second inequality is also clearly true for those parameter ranges.
\end{proof}

\begin{lem}\label{lem:tilde_estimate}
	We have the inequalities
	\begin{equation}\label{eq:plambda}
		\Big(\frac{ 2 }{1+\eta}\Big)^{-3\lambda} \leq 
	\frac{\lambda + \varepsilon}{\eta\lambda + \varepsilon} \leq \Big(\frac{ 2 }{1+\eta}\Big)^{2\lambda}
	\end{equation}
	for every $\varepsilon\in\{\pm 1\}$, every $\eta\in (-1,1)$, and every $\lambda\in(0,1)$.

\end{lem}
\begin{proof}
	For fixed $\eta\in (-1,1)$ and $\lambda>0$,
	\[
		\frac{\lambda - 1}{\eta\lambda -1}	= \frac{1-\lambda}{1-\eta\lambda}\leq \frac{1+\lambda}{1+\eta\lambda}.
	\]
	This shows that the right-hand inequality for $\varepsilon = 1$ also implies the right-hand inequality for $\varepsilon =-1$.
	Additionally, the left-hand inequality for $\varepsilon = -1$ implies the left-hand inequality for $\varepsilon = 1$.
	
	We now show the right-hand inequality and assume without restriction that $\varepsilon = 1$. We can get equality in \eqref{eq:plambda}
	if the exponent depends also on $\eta$ by taking the logarithm of \eqref{eq:plambda}. Let the corresponding function of 
	the exponent be $p(\lambda,\eta)$,
	which is then given by 
	\[
		p(\lambda,\eta) = \frac{\log(1+\lambda) - \log(1+\eta\lambda)}{\log 2 - \log(1+\eta)}.
	\]
	We now show that $p(\lambda,\eta) \leq p(\lambda):=2\lambda$ for all $\eta\in (-1,1)$.
	Indeed, we can write 
	\[
		p(\lambda,\eta) = \lambda \frac{\int_{\eta}^1 \frac{1}{1+t\lambda}\dif t}{\int_{\eta}^1 \frac{1}{1+t} \dif t}.
	\]
	Using Lemma~\ref{lem:triv} on the integrand $1/(1+t\lambda)$ pointwise gives the asserted inequality 
	$p(\lambda,\eta)\leq 2\lambda$.  Since the right hand side of \eqref{eq:plambda} is increasing 
	in the exponent, inequality \eqref{eq:plambda} is satisfied as desired. 

	Next, we show the left-hand inequality and assume that $\varepsilon =-1$.
	Similarly, the exponent $p(\lambda,\eta)$ takes the form 
	\[
		p(\lambda,\eta) = \frac{\log(1-\lambda) - \log(1-\eta\lambda)}{\log 2 - \log(1+\eta)} = -\lambda
		\frac{\int_{\eta}^1 \frac{1}{1-t\lambda}\dif t}{\int_{\eta}^1 \frac{1}{1+t} \dif t}.
	\]
	By Lemma~\ref{lem:triv} again, we see that $p(\lambda,\eta) \geq -3\lambda$, implying the assertion.
\end{proof}

Having defined the densities $d_n$ for all $n\in \mathbb N$, and given an atom 
$B\in \mathscr A_j$, we define 
\[
	d(B) := d_j(B).	
\] 
By construction of the functions $(d_n)$, this is well defined. 
By formula \eqref{eq:dnprime} and Lemma~\ref{lem:tilde_estimate} (with $\eta = 2\mu-1$), 
we have for any atom $B\in\mathscr A$ 
\begin{equation*}
d(B) = d(\pp(B))	\Big(\frac{\nu(\pp(B))}{\nu(B)}\Big)^{\alpha},
\end{equation*}
for some $\alpha\in [-3\lambda,2\lambda]$,
and inductively for atoms $A\subset B$
\begin{equation}\label{eq:bd_d}
\Big(\frac{\nu(A)}{\nu(B)}\Big)^{3\lambda}\leq \frac{d(A)}{d(B)} 
\leq \Big(\frac{\nu(A)}{\nu(B)}\Big)^{-2\lambda}.
\end{equation}
We consider the measures 
$\dif\theta_n =g^2 \dif\tilde{\mathbb P}_n =  g^2 d_n \dif\mathbb P $.
Define 
\[
\theta(A) := \lim_{n\to\infty} \theta_n(A),\qquad A\in\cup_j \mathscr A_j,
\] 
where the limit exists since for fixed $A$, the sequence $\theta_n(A)$ is eventually constant by \eqref{eq:compatibility}.
The set function $\theta$ is an additive measure on $\cup_j \mathscr A_j$, 
We have
for each atom $B$, by the above estimate for $d(B)$,
\begin{equation}\label{eq:theta_estimate}
	\theta(B) =  \int_B g^2 d(B) \dif\mathbb P \leq  \nu(B)^{1-2\lambda}.
\end{equation}
Since $\nu$ is a $\sigma$-additive, also $\theta$ is $\sigma$-additive ($\emptyset$-continuity).
Therefore it extends to a $\sigma$-additive measure on the $\sigma$-algebra generated 
by $\cup_j \mathscr A_j$.
Since functions $f\in C(\Omega)$ can be approximated by simple functions with respect 
to the atoms $\mathscr A$ (by the assumption that the maximal diameter of atoms 
tends to zero), we deduce that $\theta_n$ converges also weakly* to $\theta$.

By the assumption $g\geq c_3 > 0$, we see that  
\[
\tilde{\mathbb P}_n(\Omega) \leq c_3^{-2} \int_\Omega g^2 \dif\tilde{\mathbb P}_n = c_3^{-2 }\theta_n(\Omega) \leq  c_3^{-2},
\]
which gives that $(\tilde{\mathbb P}_n)$ is bounded and therefore there exists a measure $\tilde{\mathbb P}$
which is the weak* limit of some subsequence of $(\tilde{\mathbb P}_n)$.
Since by definition $\dif \theta_n = g^2 \dif\tilde{\mathbb P}_n$, we obtain $d\theta = g^2 \dif\tilde{\mathbb P}$.
Note that on $A\in \mathscr A_{n-1}$, we have  $\Delta\tilde{f}_n = a g\charfun_{A'} + b g\charfun_{A''}$ 
for some numbers $a,b$. Therefore
\begin{align*}
	\int_A \Delta\tilde{f}_n g\dif\tilde{\mathbb P}	&= a\int_{A'} g^2\dif\tilde{\mathbb P} + b \int_{A''} g^2\dif\tilde{\mathbb P} 
	= a \theta(A') + b \theta(A'') \\
	&= a \theta_n(A') + b\theta_n(A'') 
	= a \int_{A'} g^2 d_n \dif\mathbb P + b \int_{A''} g^2 d_n \dif\mathbb P \\
	&= \int_A \Delta \tilde{f}_n g d_n \dif\mathbb P = \int_A \Delta \tilde{f}_n g\dif\mathbb P_n  = 0,
\end{align*}
by condition \eqref{eq:cond}. This gives that $\Delta\tilde{f}_n$ is a difference sequence with 
respect to the measure $\tilde{\mathbb P}$. 

Summarizing, we have the following result.
\begin{thm}\label{thm:dim1}
	For $\lambda\in [0,1)$ and $\dim S = 1$, the functions $(\tilde{f}_n)$ form a smartingale 
	with respect to a probability measure $\tilde{\mathbb P}$
satisfying that  for each pair of atoms $A\subset B$
there exists a number $\alpha\in [-2\lambda, 3\lambda]$ such that 
\[
c_3^{2+3\lambda}\Big(\frac{| A|}{|B|}\Big)^{1+3\lambda}	
\leq\frac{\tilde{\mathbb P}(A)}{\tilde{\mathbb P}(B)} 
\leq c_3^{2\lambda-2} \Big(\frac{|A|}{|B|}\Big)^{1-2\lambda},
\]	
with the constant $c_3$ satisfying $g\geq c_3$.
\end{thm}
\begin{proof}
	The only thing that is still left to show is the given estimate. Since $d(A) = \theta(A)/\nu(A)$
	for atoms $A$, estimate \eqref{eq:bd_d} shows the desired estimate 
	with $\theta$ in place of $\tilde{\mathbb P}$ and with $\nu$ in place of $|\cdot|$.
	The assertion then follows from the inequalities $c_3\tilde{\mathbb P}(B) \leq \theta(B)\leq \tilde{\mathbb P}(B)$
	and $c_3  |B| \leq \nu(B) \leq |B|$ for atoms $B$.
\end{proof}

\subsection{The case of $\dim S>1$}\label{sec:higher_dim_S}

In this section, we treat the more general case of $\dim S > 1$.
For notational simplicity, we assume that the constant function $g = 1$ is 
contained in $S$. 
We begin with the following simple result which is a consequence of 
Remez' inequality \eqref{eq:remez}.

\begin{lem}\label{lem:product}
		 For all $f,u\in S$ and all atoms $A\in\mathscr A$ we have  for each $1\leq p\leq\infty$
		\[
			\|f\|_A \|u\|_{L^p(A)} \lesssim \|u f\|_{L^p(A)}.
		\]
\end{lem}
\begin{proof}
	Let $\alpha=(4C)^{-n}$ where $n,C$ are the constants in inequality \eqref{eq:remez}.
	Fix an atom $A\in\mathscr A$ and
	let $E_f = \{\omega\in A: |f(\omega)| < \alpha \|f\|_A\}$. Then, inequality \eqref{eq:remez}
	implies $|E_f|\leq |A|/4$ and $|E_u|\leq |A|/4$. This gives for $B=(A\setminus E_f)\cap (A\setminus E_u)$
	that  $|B| \geq |A|/2$ and on $B$ we have $|f|\geq \alpha \|f\|_A$ and $|u|\geq \alpha\|u\|_A$.
	This gives $|uf| \geq \alpha^2 \|u\|_A \|f\|_A$ on $B$ and thus the assertion for $p=\infty$.
	Since all the (renormalized) $p$-norms of polynomials on $A$ are comparable, the assertion 
	follows for all $1\leq p\leq \infty$.
\end{proof}

We continue with assuming that the atoms $\mathscr A$ are given such that the \emph{siblings in $\mathscr A$ have
comparable (Lebesgue) measure}, i.e. there exists a constant $c>0$ such that for all atoms $A\in \mathscr A$
we have $|A'| \geq c|A''|$.
Then, $A'$ and $A''$ also have comparable diameter by our assumption $|A|\simeq (\diam A)^d$
for atoms $A$.

\begin{lem}\label{lem:existence_ell}
	Assume that the siblings in $\mathscr A$ have comparable measure.

	Then there exists $\ell_0 \geq 0$ such that for all $\ell\geq \ell_0$
	there exists a constant  $c>0$  
	 such that for all $j\geq 0$ and atoms $A\in\mathscr A_j$ 
and all $u\in S^2$
\begin{equation*}
	\sum_{B\in \mathscr A_{j+\ell}, B\subseteq A} \Big| \int_{B} u\dif\mathbb P\Big| \geq c\int_A |u|\dif\mathbb P.	
\end{equation*}
\end{lem}
\begin{proof}
	Let $\|u \|_A = 1$ and let $y\in A$ be such that $|u(y)|=1$. For any $B\subseteq A$ with $y\in B$
	we have by inequality \eqref{eq:lip_condition} 
	\[
		|u(z)| \geq 1 - C\frac{\diam B}{\diam A},\qquad z\in B,
	\]
	for some constant $C$.
	Since the siblings in $\mathscr A$ are comparable, we can choose the integer 
	$\ell_0$ independently of $A$ so that for all $\ell\geq \ell_0$ and all $B\in\mathscr A_{j+\ell}$ 
	with $B\subseteq A$ satisfy 
	$C\diam B / \diam A \leq 1/2$. Thus, if $B\in \mathscr A_{j+\ell}$ is such that 
	$y\in B$ we obtain 
	\[
		\Big| \int_B u\dif\mathbb P \Big|\geq |B|/2.	
	\]
	This implies the assertion since the measures of all $B\in\mathscr A_{j+\ell}$ with $B\subset A$
	are comparable with a constant that depends only on $\ell$.
\end{proof}

\begin{lem}\label{lem:independent}
	Assume that the siblings in $\mathscr A$ have comparable measure.
	
	Let $A\in\mathscr A_j$ be an atom and let $f\in S$ and $q\in S^2$. Moreover, let $h\in S(A')\oplus S(A'')$
	with $h \perp S(A)$. 
	Choose $\ell \geq \ell_0$ with $\ell_0$ as in Lemma~\ref{lem:existence_ell}.
	Let $(A_i)_{i=1}^{2^\ell}$ be an enumeration 
	of the atoms from $\mathscr A_{j+\ell}$ contained in $A$ and define 
	\[
		v_i = \int_{A_i} h f\dif\mathbb P, \qquad	w_i = \int_{A_i} q\dif\mathbb P,\qquad i=1,\ldots 2^\ell.
	\]

	Then, there exists a constant $\alpha<1$ independent of the choices of $A,f,q,h$
	such that we have the inequality 
	\begin{equation}\label{eq:non_parallel}
			|\langle v,w\rangle| \leq \alpha |v| |w|.
	\end{equation}
\end{lem}
\begin{proof}
	We assume without restriction that $\|f\|_A=1$ and $|v| = |w| = 1$.	 Then we first show 
	\begin{equation}\label{eq:positive}
		\int_A |hf - q|\dif \mathbb P \geq \beta	
	\end{equation}
	for some constant $\beta >0$.

	Assume the contrary, i.e. for some $\varepsilon >0$ we have $\int_A |hf - q|\dif\mathbb P\leq\varepsilon$.
	Let $h = h_1\charfun_{A'} + h_2\charfun_{A''}$ for some functions $h_1,h_2 \in S$.
	Since the siblings in $\mathscr A$ have comparable measure, inequality \eqref{eq:remez} implies $\int_A |h_1 - h_2| |f|\dif\mathbb P\lesssim \varepsilon$.
	Thus, by Lemma~\ref{lem:product}  we have $|A|\| h_1 - h_2\|_A \simeq \int_A |h_1 - h_2|\dif\mathbb P \lesssim \varepsilon$.
	Therefore, by orthogonality
	\begin{align*}
		\int_A h^2\dif\mathbb P &= \int_{A} h\cdot h_2\dif\mathbb P+ \int_{A'} h(h_1 - h_2)\dif\mathbb P \\
		&= \int_{A'} h(h_1 - h_2)\dif\mathbb P \leq \Big(\int_{A'} h^2 \dif\mathbb P\Big)^{1/2} 
		\Big( \int_{A'} |h_1 - h_2|^2\dif\mathbb P\Big)^{1/2} \\
		&\lesssim \Big(\int_{A} h^2 \dif\mathbb P\Big)^{1/2} |A|^{1/2} \|h_1 - h_2\|_A,
	\end{align*}
	giving 
		$\|h\|_{L^2(A)} \lesssim \varepsilon |A|^{-1/2}$.
	On the other hand, by the assumption $|v|=1$ we obtain by Cauchy-Schwarz
	\[
		1 = \sum_{i=1}^{2^\ell}\Big(\int_{A_i} hf\dif\mathbb P\Big)^2	\leq \sum_{i=1}^n \int_{A_i} h^2\dif\mathbb P 
		\int_{A_i} f^2\dif\mathbb P \leq \sum_{i=1}^n |A_i| \int_{A_i} h^2 \dif\mathbb P\lesssim |A|\int_A h^2\dif\mathbb P
	\]
	contradicting the above inequality for sufficiently small $\varepsilon$. Thus we have proved 
	\eqref{eq:positive}.

	Now we just apply Lemma~\ref{lem:existence_ell} to get $|v-w| \gtrsim 1 = |v| |w|$. 
	Applying this inequality to $v$ and $-v$ (or $f$ and $-f$) yields the assertion \eqref{eq:non_parallel}
	if we use the identity $|v-w|^2 = |v|^2 + |w|^2 - 2\langle v,w\rangle = 2(1-\langle v,w\rangle)$.
\end{proof}

\begin{lem}\label{lem:det_est}
	Let $v^1,\ldots,v^n \in \mathbb R^n$ be $n$ vectors satisfying $|v^i| = 1$ for all $i = 1,\ldots,n$.
	Additionally, assume that for all $i = 1,\ldots,n-1$ and for all vectors $v$ in the span 
	of $\{v^{i+1},\ldots,v^n\}$ we have the inequality 
	\[
		|\langle v^i,v\rangle|\leq \alpha |v|	
	\]
	for some constant $\alpha < 1$.

	Then we have 
	\[
		|\det(v^1,\ldots,v^n)| \geq (1-\alpha^2)^{n/2}.	
	\]
\end{lem}
\begin{proof}
	We apply Gram-Schmidt (without normalization) to the vectors $(v^1,\ldots,v^n)$ resulting in 
	orthogonal vectors $(\tilde{v}^1,\ldots,\tilde{v}^n)$ with the same determinant 
	as the original vectors.
	Fix $i\in \{1,\ldots,n-1\}$. Then we have, for some $w \in \operatorname{span}\{ v^{i+1},\ldots,v^n\}$ 
	that $\tilde{v}^i = v^i - w$. 
	Then we estimate by assumption
	\begin{align*}
		|\tilde{v}^i|^2 = |v^i - w|^2 &= |v^i|^2 - 2\langle v^i,w\rangle +|w|^2 \geq |v^i|^2 - 2\alpha|w| +  |w|^2 \\
		& = 1-\alpha^2 + \alpha^2 -2\alpha|w|+|w|^2 = 1-\alpha^2 + (\alpha- |w|)^2 \geq 1-\alpha^2.
	\end{align*}
	Since the vectors $\tilde{v}^i$ are orthogonal, we have 
	\[
		|\det(v^1,\ldots, v^n)| = |\det(\tilde{v}^1,\ldots,\tilde{v}^n)| = |\tilde{v}^1| \cdots |\tilde{v}^n|
		\geq (1-\alpha^2)^{n/2},
	\]
	proving the lemma.
\end{proof}

\begin{defin}\label{def:sparse}
A smartingale $(f_n)$ with respect to the filtration $(\mathscr F_n)$
is called \emph{$\nu$-sparse} if $\Delta f_n \neq 0$ and $\Delta f_m\neq 0$ 
either imply that $n=m$ or $|n-m|\geq \nu+1$.
\end{defin}

We choose and fix the positive integer $\ell$ to be the smallest integer  
such that $\ell\geq \ell_0$ (with $\ell_0$ from Lemma~\ref{lem:existence_ell}) and $2^{\ell+1} \geq \dim S + \dim S^2$.

\begin{thm}\label{thm:tilde_P}
	Assume that the siblings in $\mathscr A$ have comparable measure.

	For each $\ell$-sparse smartingale $(f_n)$ (with respect to Lebesgue measure $\mathbb P$)
	 and each sufficiently small $\lambda >0$,
	there exists a probability measure $\tilde{\mathbb P}$ such that $(\tilde{f_n})$
	is a smartingale with respect to $\tilde{\mathbb P}$.
	Additionally, there exists a constant $c>0$ such that 
	 the measure $\tilde{\mathbb P}$ satisfies 
 for each pair of atoms $B\subset A$ the inequalities
\begin{equation}\label{eq:measure_ineq_dim_gtr_1}
\Big(\frac{| B|}{|A|}\Big)^{1+c\lambda}	
\leq\frac{\tilde{\mathbb P}(B)}{\tilde{\mathbb P}(A)} 
\leq  \Big(\frac{|B|}{|A|}\Big)^{1-c\lambda}.
\end{equation}
\end{thm}
\begin{proof}
	Let $(i_n)$ be the increasing sequence of numbers such that $\Delta f_{i_n} \neq 0$
	satisfying $i_n - i_{n-1} \geq \ell+1$.
We let $\tilde{\mathbb P}_0 = \mathbb P$ and assume that $\dif \tilde{\mathbb P}_{n-1} = d_{n-1}\dif\mathbb P$ is 
already defined with a density $d_{n-1}$ that is $\mathscr F_{i_n-1}$-measurable.

	Let $A$ be an arbitrary atom of $\mathscr F_{i_n-1}$, 
	 $m = 2^{\ell+1}$ and let $(A_j)_{j=1}^{m}$ be an 
	enumeration of the atoms of $\mathscr F_{i_{n} + \ell}$ that are contained 
	in $A$. 
We now define $\dif\tilde{\mathbb P}_n = d_n \dif\mathbb P$, restricted to the atom $A$
of $\mathscr F_{i_n-1}$
 with the density $d_n$ being of the form 
\begin{equation}\label{eq:dn_density}
	d_n\charfun_A = \sum_{j=1}^{m} d_n^{(j)} \charfun_{A_j},
\end{equation}
for some numbers  $(d_n^{(j)})_{j=1}^m $.
Those numbers are chosen such that the conditions \eqref{eq:cond} and \eqref{eq:compatibility} are satisfied, which we reproduce here for convenience:
\begin{align}\label{eq:main_conditions}
 \int_A \Delta\tilde{f}_{i_n} \cdot f \dif\tilde{\mathbb P}_n = \sum_{j=1}^m d^{(j)}_n\int_{A_j} \Delta\tilde{f}_{i_n} \cdot f \dif\mathbb P &= 0, \qquad f\in S, \\
\label{eq:comp_conditions} \int_A q \dif\tilde{\mathbb P}_n = \sum_{j=1}^m d^{(j)}_n \int_{A_j} q\dif\mathbb P &=
 \int_A q \dif\tilde{\mathbb P}_{n-1} , \qquad q\in S^2.
\end{align}
We analyze this linear system of equations first in the case when $\lambda =0$, in which we 
know the solution $\tilde{\mathbb P}_n = \tilde{\mathbb P}_{n-1}$, or equivalently, for all $j=1,\ldots,m$ 
we have that the value of $d_n^{(j)}$ is the same as the (constant) value $d_{n-1}(A)$ of $d_{n-1}$
on the atom $A$ of $\mathscr F_{i_n-1}$.
To this end, we choose suitable bases of $S(A)$ and $S^2(A)$.
Writing $h = \Delta f_{i_n}$ and $k=\dim S(A)$, we choose an orthonormal basis $b_1,\ldots,b_k$
of $S(A)$ with respect to the bilinear form 
	\begin{equation}\label{eq:bilinear1}
		(u,v)\mapsto \sum_{j=1}^m \big(\int_{A_j} hu \dif\mathbb P\big) \big(\int_{A_j} hv \dif\mathbb P\big).
	\end{equation}
	Next, with $k' = \dim S^2(A)$, choose an orthonormal basis $q_1,\ldots,q_{k'}$ of $S^2(A)$
	with respect to the bilinear form 
	\begin{equation}\label{eq:bilinear2}
		(u,v)\mapsto \sum_{j=1}^m \big(\int_{A_j} u \dif\mathbb P\big) \big(\int_{A_j} v \dif\mathbb P\big).
	\end{equation}
	The choices of the orthonormal bases $b_1,\ldots,b_k$ and $q_1,\ldots,q_{k'}$ are possible since both 
	bilinear forms in \eqref{eq:bilinear1} and \eqref{eq:bilinear2} are positive definite by
	Lemma~\ref{lem:existence_ell}.

Define the matrices 
\[
T_1 = \Big(\int_{A_j} h b_i \dif\mathbb P\Big)_{ 1\leq i\leq k, 1\leq j\leq m},\qquad 
T_2 = \Big( \int_{A_j} q_i \dif\mathbb P \Big)_{ 1\leq i\leq k',1\leq j\leq m}.
\]
Then, with the vector $d = (d_n^{(1)},\ldots, d_n^{(m)} )$ and for $\lambda=0$,
 equation \eqref{eq:main_conditions}  is equivalent to the equation 
$T_1 d = 0$ and equation \eqref{eq:comp_conditions} is equivalent to  $T_2 d = r$
with $r_i = \int_A q_i \dif\tilde{\mathbb P}_{n-1}=d_{n-1}(A) \int_A q_i \dif\mathbb P$ for $i=1,\ldots ,k'$.
Then we let $T$ be the $m\times m$ matrix consisting of all the rows from $T_1$ and $T_2$
and possibly additional arbitrary rows that are orthonormal to the rows in $T_1$ and $T_2$.
Note that $T$, applied to the constant vector $d_{n-1}(A)\cdot (1,\ldots,1)^T$ gives $(0,r,z)^T$
for some $z\in \mathbb R^{m-k-k'}$ so that  $|(0,r,z)|\lesssim d_{n-1}(A)$.
Now, combining Lemma~\ref{lem:independent} and 
Lemma~\ref{lem:det_est} gives the inequality 
\[
		|\det T| \geq \alpha' := (1-\alpha^2)^{m/2} > 0.
\]
Since all the rows in $T$ are normalized, we have for all square submatrices $T'$ of $T$ that 
\[
		|\det T'| \leq 1.
\]
This gives a uniform upper bound of $\|T^{-1}\|$, since  
the entries of $T^{-1}$ are quotients of a minor of $T$ and 
the determinant of $T$, multiplied with some sign.

Now we look at the changes of the just defined linear system of 
equations when $\lambda >0$. The changes concern 
the matrix $T_1$. 
Set
	\[
		\tilde T_1 = T_1 - \lambda \Big( \int_{A_j} p_n b_i \dif\mathbb P \Big)_{1\leq i\leq k, 1\leq j\leq m}.
	\]
	Moreover, let $\tilde{T}$ be the matrix $T$ with $T_1$ exchanged by $\tilde{T}_1$.
	Then, equation \eqref{eq:main_conditions} is equivalent to $\tilde T_1 d = 0$ and 
	\eqref{eq:comp_conditions} is equivalent to $T_2 d = r$.

	We now estimate the difference $\Delta T = T - \tilde{T}$. Since by definition, the 
	functions $b_i$ satisfy $1 = \sum_{j=1}^m \big( \int_{A_j} h b_i \dif\mathbb P\big)^2$, 
	we obtain 
	\begin{equation}\label{eq:lower_est_h}
				1 \gtrsim \sum_{j=1}^m \Big | \int_{A_j} h b_i \dif\mathbb P\Big| 
				\gtrsim \sum_{B\in \{A',A''\}} \int_B |h b_i| \dif\mathbb P
				\gtrsim \sum_{B\in\{A',A''\}} \|h\|_B \|b_i\|_{L^1(B)},
	\end{equation}
	where we used Lemmas~\ref{lem:existence_ell} and \ref{lem:product}, respectively. On the other 
	hand, for each $i=1,\ldots,k$,
	\begin{align*}
		\sum_{j=1}^m |(\Delta T)_{ij}| &= \lambda \sum_{j=1}^m \Big| \int_{A_j} p_n b_i\dif\mathbb P\Big|
		\leq \lambda\sum_{j=1}^m \int_{A_j} |p_n b_i|\dif \mathbb P =\lambda \sum_{ B\in \{A',A''\}} \int_B |p_nb_i|\dif\mathbb P \\
		&\leq\lambda \sum_{B\in\{A',A''\}} \|p_n\|_B \|b_i\|_{L^1(B)} \lesssim \lambda\sum_{B\in \{A',A''\}} \|h\|_B \| b_i\|_{L^1(B)}.
	\end{align*}
	This gives, together with \eqref{eq:lower_est_h},
	\begin{equation}\label{eq:Delta_bound}
		\| \Delta T\| \lesssim \lambda \|T\|.
	\end{equation}
	Since $\tilde{T}^{-1} = (\operatorname{Id} - T^{-1}\Delta T)^{-1} T^{-1}$ and by the 
	uniform upper bound of $\|T^{-1}\|$, inequality \eqref{eq:Delta_bound} implies that 
	for $\lambda$ sufficiently small we have $\|T^{-1}  - \tilde{T}^{-1}\| \lesssim \lambda$. 
	The vector $d = \tilde{T}^{-1} (0,r,z)^T$ satisfies conditions \eqref{eq:main_conditions}
	and \eqref{eq:comp_conditions} and we know that $T^{-1} (0,r,z)^T = d_{n-1}(A)(1,\ldots,1)^T$,
	which gives 
	\begin{equation}\label{eq:error_estimate}
		| d - d_{n-1}(A)(1,\ldots,1) | = | (\tilde{T}^{-1} - T^{-1})(0,r,z)^T | 
		\lesssim \lambda | (0,r,z)| \lesssim \lambda\cdot d_{n-1}(A).
	\end{equation}
	For sufficiently small $\lambda$ we therefore obtain that $d$ consists of positive entries 
	resulting in a positive measure $\tilde{\mathbb P}_n$.
	Since $(f_n)$ is $\ell$-sparse, we obtain that the 
	density $d_n$ of $\tilde{\mathbb P}_n$ with respect to $\mathbb P$ 
	 is $\mathscr F_{i_{n+1}-1}$-measurable.

	We now come to the estimates. We have 
	$\tilde{\mathbb P}_n(A) = \tilde{\mathbb P}_{n-1}(A)= d_{n-1}(A) |A|$ and $\tilde{\mathbb P}_n(A_j) = d_j |A_j|$
	by condition \eqref{eq:comp_conditions} and the fact that the constant function is contained in $S$.
	Using \eqref{eq:error_estimate}, this gives for all $B=A_j$ and some constant $C'$,
	\begin{equation}\label{eq:measure_ineq}
		 \frac{\tilde{\mathbb P}_n(B)}{|B|} =d_j \leq (1+C'\lambda)d_{n-1}(A) =  (1+C'\lambda)\frac{\tilde{\mathbb P}_n (A)}{|A|}.
	\end{equation}
	Since siblings in $\mathscr A$ have comparable measure, there exists a constant 
	$R > 1$ such that $|A| / |B| \geq R$ and the elementary inequality $e^t \geq 1+t$ gives 
	some constant  $c>0$ such that $1+C'\lambda \leq (|A| / |B|)^{c\lambda}$.
	Together with \eqref{eq:measure_ineq} this yields the right hand side of 
	\eqref{eq:measure_ineq_dim_gtr_1}.
	The left hand side of \eqref{eq:measure_ineq_dim_gtr_1} is proved similarly by using 
	the inequality $e^t \leq (1-t)^{-1}$ for $t<1$.

	 From the sequence $(\tilde{\mathbb P}_n)$ we now pass to the limit
	 as in the proof of the case $\dim S = 1$ to 
	  obtain the desired measure $\tilde{\mathbb P}$ satisfying the same estimates
	  as $\tilde{\mathbb P}_n$.
\end{proof}

\begin{rem} We remark that the measure $\tilde{\mathbb P}$ in Theorem~\ref{thm:tilde_P} satisfies 
inequality \eqref{eq:L1Linfty} with the same space $S$ and the same atoms $\mathscr A$ 
as the Lebesgue measure $\mathbb P$.
In order to see this we fix an atom $A\in\mathscr A_j$ for some $j$ and a function $f\in S$.
Arguing as in the proof of Lemma~\ref{lem:existence_ell}, there exists an atom $B\in \mathscr A_{j+\ell}$
with $B\subset A$ such that $|f| \geq \|f\|_A / 2$ on $B$.
Inequality \eqref{eq:measure_ineq_dim_gtr_1} and the fact that siblings in $\mathscr A$ have 
comparable (Lebesgue) measure now give inequality \eqref{eq:L1Linfty} for the 
measure $\tilde{\mathbb P}$.
\end{rem}

\section{Martingales from smartingales}\label{sec:lil_smart}
In this section we are going to prove a LIL estimate for smartingales
on a probability space $(\Omega,\mathscr F, \mu)$ satisfying \eqref{eq:L1Linfty} 
by comparing them to certain martingales. 
Our applications in Section~\ref{sec:appl} will be concerned with using those results with respect to the 
measures $\tilde{\mathbb P}$ constructed in Section~\ref{sec:tilde_measure}.

For martingales $(M_n)$ and a filtration $(\mathscr F_n)$, we have Freedman's inequality \cite{Freedman1975} 
 about their square functions
$S_n^2 = \sum_{j\leq n} \mathbb E_{j-1} |\Delta M_j|^2$, which reads as follows.
\begin{lem}\label{lem:mart_exp}
	Let $(M_n)$ be a real martingale with bounded differences $|\Delta M_n|\leq L$ and  $ M_0 = 0$. For $a>0$, let 
	\[
		\tau_a = \inf\{ n : M_n \geq a\}.	
	\]
	Then, for any $b>0$, we have 
	\[
		\mu( S_{\tau_a}^2 \leq b, \tau_a < \infty) \leq \exp\Big( -\frac{a^2}{2(L a+b)}\Big).
	\]
\end{lem}

An inequality of this type implies a (LIL type) upper estimate of the following sort
if $R_n=M_n$ and $Q_n = S_n^2$. In fact, it is not essential here that we consider 
martingales and their square functions, but only random variables $R_n, Q_n$ that 
are related by an inequality of the type \eqref{eq:axiom_exp} below.
\begin{lem}\label{lem:exp_implies_lil}
	Let $(R_n)$ a sequence of real random variables and 
	 $(Q_n)$ a  sequence of real random variables adapted to $(\mathscr F_n)$ 
	with $Q_n\geq Q_{n-1}\geq 0$ for each $n$,
	 satisfying for some $r,L>0$ and for each $a,b>0$ with $a$ sufficiently large
	 \begin{equation}\label{eq:axiom_exp}
		\mu( Q_{\tau_a} \leq b, \tau_a < \infty)	\leq \exp\Big( - \frac{a^2}{r(L a+b)}\Big),
	 \end{equation}
	with $\tau_a = \inf \{ n : R_n \geq a\}$.

	Then, 
	\[
		\limsup_{n\to\infty} \frac{R_n}{\sqrt{r Q_n \log\log Q_n}}	\leq 1,\qquad \mu\text{-a.s. on } \{ Q_\infty =\infty\},
	\]
	where we write $Q_\infty = \sup_n Q_n$.
\end{lem}
The proof of this result is just a minor modification of the proof of
Theorem (6.1) in \cite{Freedman1975}.
\begin{proof}
	Let $q > 1$ and denote, for $\varepsilon > 0$,
	\[
		A = \{ Q_\infty = \infty, R_n > q \sqrt{(r+\varepsilon) Q_n\log\log Q_n}\text{ for infinitely many $n$}	\}.
	\]
	Introduce the $(\mathscr F_n)$ stopping times 
	\[
		\sigma_k = \inf \{ n : Q_n \geq q^k\}
	\]
	and the events 
	\[
		A_k = \Big\{ Q_\infty = \infty, \text{there exists } n\in [\sigma_k,\sigma_{k+1}) : R_n > q \sqrt{(r+\varepsilon) Q_n\log\log Q_n} \Big\}.
	\]
	We have $A \subset \limsup_k A_k = \cap_n \cup_{k\geq n} A_k$.
	Define 
	\[
		a = q \sqrt {(r+\varepsilon)q^k \log\log q^k},\qquad b = q^{k+1}.
	\]
	This definition gives 
	\[
		A_k \subset \{ Q_{\tau_a} \leq b, \tau_a < \infty\}.
	\]
	By definition of the stopping time $\sigma_k$ and the fact that $(Q_n)$ is increasing, 
	we know that for $m < \sigma_{k+1}(\omega)$ we have $Q_m(\omega) \leq q^{k+1}$.
	Therefore, by \eqref{eq:axiom_exp}, if $k$ sufficiently large,
	\begin{align*}
		\mu(A_k) \leq \mu(Q_{\tau_a}\leq b, \tau_a < \infty) \leq 
		\exp\Big(-\frac{a^2}{r(L a+b)}\Big) \lesssim  \big(\log(q^k)\big)^{-q}
		\leq \frac{c_q}{k^{q}}
	\end{align*}
	for some constant $c_q$ independent of $k$.
	Since $q>1$, this gives 
	\[
		\mu(A) \leq \sum_{k\geq n} \mu(A_k)\to 0,\qquad \text{as }n\to\infty.	
	\]
	Letting $\varepsilon\to 0$ and $q\to 1$ yields the assertion of the lemma.
\end{proof}

The rest of this section is devoted to proving estimate \eqref{eq:axiom_exp} for smartingales 
and their square functions via its corresponding martingale result.
We begin by associating to each smartingale a certain martingale.

\begin{prop}
	Let $(f_n)$ be a smartingale, $g_n \in\operatorname{ran} P_{n-1}$ and $\varphi_n$ an $\mathscr F_{n-1}$-measurable function.

	Then, the sequence
	\[
		\Delta M_n := \mathbb E_n(g_n f_n) - \varphi_n \mathbb E_n(g_n f_{n-1}) - (1-\varphi_n)\mathbb E_{n-1}(g_n f_{n-1})
	\]
	is a martingale difference sequence with respect to the same filtration.
\end{prop}
\begin{proof}
We just have to verify that $\mathbb E_{n-1} (\Delta M_n) = 0$. Indeed, we first remark that 
\begin{equation}\label{eq:simple}
	\mathbb E_{n-1}(g_n f_n) = \mathbb E_{n-1}(g_n P_{n-1}f_n)	= \mathbb E_{n-1}(g_n f_{n-1}),
\end{equation}
where we used that $g_n\charfun_A$ is contained in the range of $P_{n-1}$
for each $A\in\mathscr A_{n-1}$, the self-adjointness of $P_{n-1}$, and  
that $(f_n)$ is a smartingale.
Then we calculate
\begin{align*}
	\mathbb E_{n-1} (\Delta M_n) &= \mathbb E_{n-1}(g_n f_n) - \varphi_n \mathbb E_{n-1}(g_n f_{n-1}) - (1-\varphi_n)\mathbb E_{n-1}(g_n f_{n-1}) \\
	&=\mathbb E_{n-1}(g_n f_n) - \mathbb E_{n-1}(g_n f_{n-1}) =0,
\end{align*}
where we used the predictability of $(\varphi_n)$ and \eqref{eq:simple}, respectively.
\end{proof}

In what follows we will use the case $\varphi_n=0$ for all $n$.
Moreover, we choose $g_n = g$ for each $n$, where $g\in S$ is the function $g : \Omega\to [c_3,1]$
with the positive constant $c_3$.
Then, by \eqref{eq:simple},
		\[
			\Delta M_n = \mathbb E_n (g f_n) - \mathbb E_{n-1}(g f_{n-1})	= (\mathbb E_n-\mathbb E_{n-1}) (g f_n)
		\]
		and therefore 
		\[
			M_n = \mathbb E_n(g f_n).
		\]
Correspondingly, for a smartingale $(f_n)$ and its associated martingale $(M_n)$, 
denote the square functions 
\begin{equation}\label{eq:def_square_functions}
	\begin{aligned}
S^2_n &= \sum_{k\leq n} (\mathbb E_k + \mathbb E_{k-1})(\Delta f_k)^2,&\qquad S_{M,n}^2 &= \sum_{k\leq n} \mathbb E_{k-1}|\Delta M_k|^2, \\
S^2 &= \sup_n S^2_n,&\qquad  S_{M}^2 &= \sup_n S_{M,n}^2.
	\end{aligned}
\end{equation}

As already mentioned, we want to show the subgaussian estimate~\eqref{eq:axiom_exp} 
for smartingales, i.e. with the choice $R_n = f_n$ and $Q_n = S_n^2$. Since this subgaussian 
estimate is known for martingales, i.e. for the choice $R_n = M_n$ and $Q_n = S_{M,n}^2$, 
the plan is the following:
\begin{enumerate}
	\item Give the pointwise inequality $S_M^2 \lesssim S^2$ between square functions.
			This is done in Theorem~\ref{thm:pw}, with the help of Lemmas~\ref{lem:bounded}, 
			\ref{lem:fat_chains}, \ref{lem:lipschitz}, and~\ref{lem:basic_estimate}.
	\item Show the inclusion $\{ \tau_{S,a} < \infty\} \subseteq \{\tau_{M,\beta a} < \infty \}$ 
		 for some positive constant $\beta$, where we put $\tau_{S,a} = \inf \{ n : f_n \geq a\}$ 
		 and $\tau_{M,a} = \inf \{ n : M_n \geq a\}$. This is performed in Proposition~\ref{prop:stop_comp}.
\end{enumerate}
Those steps are combined in Proposition~\ref{prop:exp_ineq_smart} to deduce from the 
subgaussian estimate for martingales its smartingale variant.
Applying then Lemma~\ref{lem:exp_implies_lil} gives us the upper estimate 
for the LIL for smartingales, finally summarized in Theorem~\ref{thm:lil_smartingale}.

For general smartingale differences we have the following relation between its values:
\begin{lem}\label{lem:bounded}
For smartingales $(f_n)$ we have the inequality
	\[
		\|\Delta f_n\|_{A''}	\lesssim \|\Delta f_n\|_{A'} \frac{\mu(A')}{\mu(A)}, \qquad A\in \mathscr A_{n-1}.
	\]
\end{lem}
\begin{proof}
	It is enough to show this for $\Delta f_n$ with $\|\Delta f_n\|_{L^2(A)} = 1$.
	In that case we have shown in \cite{part1} the inequalities 
	\[
		\|\Delta f_n\|_{A'}	\simeq \mu(A')^{-1/2}, \qquad \|\Delta f_n\|_{A''} \lesssim \frac{\mu(A')^{1/2}}{\mu(A)},
	\]
	which immediately give the assertion.
\end{proof}

\begin{defin}
	Let $\mathscr X=(X_i)_{i=1}^n$ be a decreasing sequence of atoms in $\mathscr A$.
	We say that $\mathscr X$ is a \emph{full chain}, if we have 
	\[
		X_{i+1} \in \{ X_i', X_i''\}\text{ for all $i=1,\ldots,n-1$}.
	\]
	In this definition, $n$ is allowed to be equal to $\infty$.

	If we additionally assume $n<\infty$ and we 
	have the condition $\mu(X_n) \geq  \mu(X_1)/2$,
	we say that 
	$\mathscr X$ is a \emph{fat full chain}.

	For a full chain $\mathscr X = (X_i)_{i=1}^n$, we denote by $\mathscr X^*$ 
	the full chain $\mathscr X^* = (X_i)_{i=2}^n$ with the first atom removed.
\end{defin}

With this terminology we see that every full chain can be decomposed 
into the union of fat full chains in the following way.
\begin{lem}\label{lem:fat_chains}
	Every full chain $\mathscr X$ can be decomposed into 
	the union of  
	fat full chains $(\mathscr X_j)$ satisfying
	\begin{equation}\label{eq:chains}
		\max_{A\in\mathscr X_{s+1}}\mu(A) < \frac{1}{2}\max_{A\in\mathscr X_{s}} \mu(A),\qquad
		s = 1,\ldots,k-1.	
	\end{equation}
\end{lem}
\begin{proof}
	Let $\mathscr X = (X_i)_{i=1}^\infty$ be a full chain. 
	Let $i_0 = 1$ and let inductively
	\[
		i_k = \min \Big\{ n>i_{k-1} : \mu(X_n) < \frac{1}{2}  \mu(X_{i_{k-1}})\Big\}.
	\]
	The chains $\mathscr X_k = (X_i)_{i=i_{k-1}}^{i_k - 1}$ obviously are 
	fat and satisfy \eqref{eq:chains}.
\end{proof}
The above arguments work for arbitrary probability spaces $(\Omega,\mathscr F,\mu)$
with a binary filtration $(\mathscr F_n)$ and a space $S$ satisfying \eqref{eq:L1Linfty}.
From now on we again assume conditions (1)--(3) from the introduction.
	Using then Markov's inequality \eqref{eq:lip_condition} we see that the following result ist true.

\begin{lem}\label{lem:lipschitz}
	Let $B\subset A$ be two atoms with $B\in\mathscr A_{n}$.
	Then we have 
	for each function $f\in S(A)$ 
	and every $C\in \{ B', B''\}$ 
	\[
		|(\mathbb E_{n+1} - \mathbb E_{n})f| 
		\lesssim \frac{\mu(B\setminus C)\operatorname{diam} B}{\mu(B) \diam A} \|f\|_{A} \qquad \text{on } C.
	\]
\end{lem}

For the decomposition of a full chain $\mathscr X$ 
into a sequence of fat full chains $(\mathscr X_j)$ 
according to Lemma~\ref{lem:fat_chains},
we denote by $L_j$ and $X_j$ the atoms with largest and smallest measure in the chain $\mathscr X_j$,
respectively. 
Below, we are going to estimate the terms appearing in the previous lemma 
after a summation over the index $n$. In order to get the desired result, 
we have to assume the following geometric decay inequality for 
the diameters of the atoms appearing in the decomposition of full 
chains $\mathscr X$ into fat chains: 
	\begin{equation}\label{eq:diam_ineq}
		\diam(L_\ell) \sum_{j\leq \ell}	(\diam X_j)^{-1} \lesssim 1.
	\end{equation}

	Note that if $\mu$ is the Lebesgue measure $|\cdot|$, inequality \eqref{eq:diam_ineq} is satisfied 
	by the assumption $(\diam A)^d \simeq |A|$ and the geometric decay of $|X_j|$ or $|L_j|$
	guaranteed by the construction of fat chains.
	We also remark that \eqref{eq:diam_ineq} implies a geometric decay of the diameters of 
	sets in fat chains, and therefore also the inequality $\sum_{j\geq \ell} \diam L_j \lesssim \diam X_\ell$.

Using \eqref{eq:diam_ineq} and the inequality from Lemma~\ref{lem:lipschitz} 
on the differences $\Delta f_k$ separately, we obtain 
the following result.

\begin{lem}\label{lem:basic_estimate}
	Assume that $f\in S_n$, $B\in \mathscr A_n$ and that \eqref{eq:diam_ineq} is true.
	Let $\mathscr X=(C_i)_{i=0}^n$ be the finite full chain
	with $C_0 = \Omega$ and $C_n = B$. 

	If $p\in [1,\infty)$, the following inequality is true on the set $B=C_n$:
	\begin{align*}
		| (\mathbb E_{n+1}-\mathbb E_n)(gf)|^p  \lesssim  \\
		\sum_{j} \frac{\diam C_n}{\diam X_j} &\frac{\mu(C_n')\charfun_{C_n''} + \mu(C_n'')\charfun_{C_n'}}{\mu(C_n)} 
			 \Big( \|\Delta f_{\ell_j}\|_{L_j}^p + \sum_{k\leq n: C_k\in\mathscr X_j^*} 
		\frac{\mu(C_{k-1}')}{\mu(C_{k-1})}\|\Delta f_k\|_{C_{k-1}'}^p \Big),
	\end{align*}
	where $\ell_j$ is such that $L_j = C_{\ell_j}$.
\end{lem}
\begin{proof}
	Decompose the function $f = f_n$ into its differences $f = \sum_{k\leq n} \Delta f_k$.
	We first estimate 
	$(\mathbb E_{n+1} - \mathbb E_{n})(g\Delta f_k)$ for fixed  $k\leq n$.
	On the set $C_n$ we have by  Lemma~\ref{lem:lipschitz} 
	\begin{equation}\label{eq:markov_appl}
	\begin{aligned}
		|(\mathbb E_{n+1} - \mathbb E_{n}) (g \Delta f_k)| \lesssim 
		\frac{\operatorname{diam} C_{n}}{\diam C_k\cdot \mu(C_{n})}
		\Big(  \mu(C_n')\charfun_{C_n''} + \mu(C_n'')\charfun_{C_n'}\Big)\|\Delta f_k\|_{C_k}.
	\end{aligned}
	\end{equation}
	Recall that we split the chain $\mathscr X$ into the fat full chains 
	$(\mathscr X_j)$  according to Lemma~\ref{lem:fat_chains}.
	For each $j$ denote by $L_j$ and $X_j$  the largest and smallest atom in 
	the chain $\mathscr X_j$, respectively.
	Then, on $C_n$, we have by \eqref{eq:markov_appl}
	\begin{align}
		\notag
		\Big(\sum_{k\leq n} |&(\mathbb E_{n+1} - \mathbb E_{n}) (g \Delta f_k)| \Big)^p
		= \Big(\sum_{j} \sum_{k\leq n : C_k\in \mathscr X_j} 
		|(\mathbb E_{n+1} - \mathbb E_{n})(g\Delta f_k)|\Big)^p \\
		\notag
		&\lesssim 
			\Big(\sum_{j}\frac{\operatorname{diam} C_{n}}{\diam X_j}
			\frac{ \mu(C_n') \charfun_{C_n''} + \mu(C_n'')
			 \charfun_{C_n'}}{\mu(C_{n})} \sum_{k\leq n : C_k\in\mathscr X_j} \|\Delta f_k\|_{C_k}\Big)^p \\
		&\lesssim  \sum_{j} \frac{\operatorname{diam} C_{n}}{\diam X_j}
		\frac{ \mu(C_n') \charfun_{C_n''} + \mu(C_n'') \charfun_{C_n'}}{\mu(C_{n})}
		\Big(\sum_{k\leq n : C_k\in\mathscr X_j} \|\Delta f_k\|_{C_k} \Big)^p,
		\label{eq:longsum}
	\end{align}
	where the latter inequality follows from \eqref{eq:diam_ineq} and Jensen's inequality.
	For fixed index $j$ 
	we split the non-empty sum over one chain
	into the parts	
	\begin{align*}
		\Big( \sum_{k\leq n:C_k\in \mathscr X_j} \|\Delta f_k\|_{C_k} \Big)^p \lesssim 
		  \|\Delta f_{\ell_j}\|_{L_j}^p + \Big( \sum_{k\leq n:C_k\in \mathscr X_j^*} \|\Delta f_k\|_{C_k} \Big)^p,
	\end{align*}
	where $\ell_j$ is such that $L_j = C_{\ell_j}$.
	Since $C_k = C_{k-1}''$ for $C_k\in \mathscr X_j^*$ we have by Lemma~\ref{lem:bounded}
	\begin{align*}
		\Big( \sum_{k\leq n:C_k\in \mathscr X_j} \|\Delta f_k\|_{C_k} \Big)^p 
		&\lesssim \|\Delta f_{\ell_j}\|_{L_j}^p + \Big( \sum_{k\leq n:C_k\in \mathscr X_j^*} 
			\frac{\mu(C_{k-1}')}{\mu(C_{k-1})} \|\Delta f_k\|_{C_{k-1}'}\Big)^p \\
		&\lesssim \|\Delta f_{\ell_j}\|_{L_j}^p +  \sum_{k\leq n:C_k\in \mathscr X_j^*} 
			\frac{\mu(C_{k-1}')}{\mu(C_{k-1})} \|\Delta f_k\|_{C_{k-1}'}^p,
	\end{align*}
	where the latter inequality is a consequence of Jensen's inequality 
	and the fact that the measures $\mu(C_{k-1})$ with $C_k\in \mathscr X_j^*$ 
	are comparable to $\mu(L_j)$. Inserting this inequality in \eqref{eq:longsum} yields the 
	conclusion.
\end{proof}

We use this inequality to derive the following pointwise result relating the martingale and smartingale square 
functions $S_M$ and $S$, respectively.
\begin{thm}\label{thm:pw}
	Suppose that $(f_n)$ is a smartingale and  \eqref{eq:diam_ineq} is true.
	
	Then, we have the following pointwise inequality between the 
	martingale square function and the smartingale square function:
	\begin{equation}\label{eq:square_comp}
		S_M^2 = \sum_n \mathbb E_{n-1}  |(\mathbb E_n - \mathbb E_{n-1})(g f_n)|^2 \lesssim 
		 \sum_n (\mathbb E_n + \mathbb E_{n-1}) |\Delta f_n|^2 =  S^2
	\end{equation}
\end{thm}
\begin{proof}
	Let $x\in\Omega$ be arbitrary and let $\mathscr X = (C_j)_{j=0}^\infty$ be the full 
	chain of atoms containing the point $x$.
	We begin by substituting $f_n = \sum_{k\leq n} \Delta f_k$. By Jensen's inequality it 
	suffices to estimate the expression 
	\begin{align*}
		\sum_n 	\mathbb E_{n-1}\Big| (\mathbb E_n &- \mathbb E_{n-1})\big(g\sum_{k< n}\Delta f_k\big)\Big|^2 
		\leq \sum_n 	\mathbb E_{n-1}\Big(\sum_{k< n} |(\mathbb E_n - \mathbb E_{n-1})(g \cdot \Delta f_k)|\Big)^2.
	\end{align*}
Using Lemma~\ref{lem:basic_estimate} with $p=2$ and taking the conditional expectation $\mathbb E_{n-1}$ 
yields on the atom $C_{n-1}$
	\begin{equation}\label{eq:cond_n-1}
		\begin{aligned}
		\mathbb E_{n-1}\Big(\sum_{k<n} |&(\mathbb E_n - \mathbb E_{n-1}) (g \Delta f_k)| \Big)^2 \\
		&\lesssim \sum_{j\leq j_n} \frac{\mu(C_{n-1}')\operatorname{diam}C_{n-1} }{\mu(C_{n-1}) \diam X_j}
		\Big(\|\Delta f_{\ell_j}\|_{L_j}^2 + \sum_{k<n : C_k\in\mathscr X_j^*}\frac{\mu(C_{k-1}')}{\mu(C_{k-1})} \|\Delta f_k\|_{C_{k-1}'} \Big),
		\end{aligned}
	\end{equation}
	with $j_n$ such that $C_n\in\mathscr X_{j_n}$.
	Summing over $n$ and rearranging the sums gives at the point $x$ the inequality
	\begin{align*}
		\sum_n \mathbb E_{n-1}\Big(\sum_{k<n} |(\mathbb E_n - \mathbb E_{n-1})&(g \Delta f_k)| \Big)^2 
		\lesssim  
		 \sum_{j=1}^\infty \|\Delta f_{\ell_j}\|_{L_j}^2 \sum_{n>\ell_j}
		 \frac{\mu(C_{n-1}')\operatorname{diam}C_{n-1}}{\mu(C_{n-1}) \diam X_j}  \\
		&+\sum_{j=1}^\infty \sum_{k: C_k\in \mathscr X_j^*} \frac{\mu(C_{k-1}')}{\mu(C_{k-1})}\|\Delta f_k\|_{C_{k-1}'}^2 
		\sum_{n>k} \frac{\mu(C_{n-1}') \operatorname{diam}C_{n-1}}{\mu(C_{n-1}) \diam X_j }.
	\end{align*}
	Inequality~\eqref{eq:diam_ineq} implies that, after decomposition into 
	fat chains $(\mathscr X_j)$,
	 the inner sums over $n$ are both $\lesssim 1$. This gives, by inequality \eqref{eq:L1Linfty},
	\begin{align*}
		\sum_n\mathbb E_{n-1}\Big(\sum_{k<n} |(\mathbb E_n - \mathbb E_{n-1})(g \Delta f_k)&| \Big)^2 
		\lesssim  \sum_{j=1}^\infty \Big( \|\Delta f_{\ell_j}\|_{L_j}^2 + 
			\sum_{k:C_k\in\mathscr X_j^*}\frac{\mu(C_{k-1}')}{\mu(C_{k-1})}\|\Delta f_k\|_{C_{k-1}'}^2\Big) \\
		&\lesssim \sum_{j=1}^\infty \Big(\mathbb E_{\ell_j} (|\Delta f_{\ell_j}|^2)(x) + 
			\sum_{k:C_k\in\mathscr X_j^*} \mathbb E_{k-1}(|\Delta f_k|^2)(x) \Big) \\
		&\lesssim \sum_n (\mathbb E_{n}+\mathbb E_{n-1}) (|\Delta f_n|^2)(x),
	\end{align*}
	completing the proof.
\end{proof}

For the smartingale $(f_n)$, recall that the associated martingale $M_n$ 
was defined by $M_n = \mathbb E_n(g f_n)$ with the function $g\in S$ satisfying $0<c_3\leq g\leq 1$.
Define for $a>0$
\[
\tau_{S,a} := \inf\{ n : f_n \geq a\},\qquad \tau_{M,a} := \inf\{ n : M_n \geq a\}.
\] 
Note that $\tau_{M,a}$ is a stopping time, but $\tau_{S,a}$ isn't.
\begin{prop}\label{prop:stop_comp}
Suppose that  $(f_n)$ is a smartingale with uniformly bounded 
differences $|\Delta f_n|\leq L$ for some constant $L$ and assume that \eqref{eq:diam_ineq} is true.

 Then, with $\beta = c_3/2$, there exists a constant $K$ so that 
 we have for all integers $n$ and 
 all positive numbers $a\geq K L$  
	\[
		\{ \tau_{S,a} = n\} \subseteq \{\tau_{M,\beta a}	\leq n\}.
	\]
\end{prop}
As a corollary, we have $\{ \tau_{S,a} < \infty \} \subseteq \{\tau_{M,\beta a} < \infty\}$ if $a>0$ 
is sufficiently large.
\begin{proof}
	Let $x\in \{ \tau_{S,a} =n\}$, which gives that $f_n(x)\geq a$.
	Let $\mathscr X = (C_j)_{j=0}^n$ be the full chain of atoms such that 
	$x\in C_j\in\mathscr A_j$ for each $j$.
	Then we estimate 
	for $y,z\in C_n$ with $y\neq z$
	\[
		\frac{|f_n(y)-  f_n(z)|}{|y-z|}  \leq |y-z|^{-1} \sum_{k\leq n} |\Delta f_k(y) - \Delta f_k(z)|
	\]
	Note that the function $\Delta f_k$ is contained in $S$ on the set  $C_k$. Therefore we use 
	Markov's inequality  \eqref{eq:lip_condition} to obtain
	\[
			\frac{|f_n(y)-  f_n(z)|}{|y-z|}	\leq C \sum_{k\leq n} \frac{\| \Delta f_k\|_{C_k}}{\diam C_k}.
	\]
	Decompose the full chain $\mathscr X$ into 
	fat chains $\mathscr X_1,\ldots, \mathscr X_\ell$, and
	denote by $L_j$ the largest and $X_j$ the smallest atom in $\mathscr X_j$, respectively.
	Then, using this decomposition, we write 
	\begin{align*}
		\sum_{k\leq n} \frac{\|\Delta f_k\|_{C_k}}{ \diam C_k} &=	
		\sum_{j=1}^\ell \sum_{k : C_k \in \mathscr X_j} \frac{\|\Delta f_k\|_{C_k}}{\diam C_k} 
		= \sum_{j=1}^\ell \Big( \frac{\|\Delta f_{\ell_j} \|_{L_j}}{\diam L_j} +
		 \sum_{k: C_k\in \mathscr X_j^*} \frac{\|\Delta f_k\|_{C_k}}{\diam C_k}\Big),
	\end{align*}
	where $\ell_j$ is such that $C_{\ell_j} = L_j$.
	For the first part, we just use the boundedness of differences $\| \Delta f_k\|_{\infty} \leq L$
	and for the second part, we invoke Lemma~\ref{lem:bounded} to get 
	\begin{align*}
		\sum_{k\leq n} \frac{\|\Delta f_k\|_{C_k}} {\diam C_k} &\lesssim
		L \sum_{j=1}^\ell \Big( \frac{1}{\diam L_j} +
		 \sum_{k: C_k\in \mathscr X_j^*} \frac{\mu(C_{k-1}') }{\mu(C_{k-1}) \diam C_k}\Big).
	\end{align*}
	Observe that for each $k$ such that $C_k \in \mathscr X_j^*$ we have 
	$C_k = C_{k-1}''$ and therefore $\sum_{k: C_k\in\mathscr X_j^*} \mu(C_{k-1}')
	 \leq \mu(L_j) \lesssim \mu(C_i)$
	for each $i$ with $C_i\in\mathscr X_j^*$.
	By assumption~\eqref{eq:diam_ineq},
	\begin{equation}\label{eq:geom_blocks}
		\sum_{k\leq n} \frac{\|\Delta f_k\|_{C_k}}{\diam C_k} \lesssim
		L\sum_{j=1}^\ell (\diam X_j)^{-1} \lesssim L  (\operatorname{diam} L_\ell)^{-1} \leq L(\operatorname{diam} C_n)^{-1}.
	\end{equation}
	Combining this with the above calculations 
	\begin{equation}\label{eq:lip}
		\sup_{y,z\in C_n, y\neq z} \frac{|f_n(y) - f_n(z)|}{|y-z|} \lesssim L (\operatorname{diam} C_n)^{-1}.
	\end{equation}
	If $D$ is the implicit constant in \eqref{eq:lip} on the right hand side
	and $a\geq 2DL$ we obtain for all $y\in C_n$
	\[
		f_n(y) \geq f_n(x) - |f_n(y)-f_n(x)| \geq a - DL \frac{|x-y|}{\diam C_n} \geq \frac{a}{2},
	\]
	by the assumption $\tau_{S,a}(x)=n$.
	Now, recall that $g\in S$ is such that $g\geq c_3$ and $\|g\|_{\Omega} = 1$.
	Therefore we conclude
	\begin{align*}
		\mathbb E_n(g f_n)(x) = \frac{1}{\mu(C_n)}\int_{C_n} g f_n \dif\mu
		\geq \frac{c_3 a }{2}
	\end{align*}
	if $a \geq 2DL$. But this immediately gives $x\in \{\tau_{M,\beta a} \leq n\}$ 
	for $a\geq 2DL$ and with $\beta = c_3/2$.
\end{proof}

Recall the definition $S_n^2 = \sum_{k\leq n} (\mathbb E_k + \mathbb E_{k-1})(\Delta f_k)^2$
of the smartingale square function $S_n$ and $S = \sup_n S_n$.

\begin{prop}\label{prop:exp_ineq_smart}
	Let $(f_n)$ be a smartingale with uniformly bounded differences $|\Delta f_n|\leq L$ 
	and assume inequality \eqref{eq:diam_ineq}.

	Then, there exists $r>0$ such that for $a>0$ sufficiently large and all $b>0$ we have 
	\[
		\mu (S^2_{\tau_{S,a}} \leq b,\tau_{S,a} < \infty)	 \leq  \exp\Big(- \frac{a^2}{r (La+b)} \Big),
	\]
	with $\tau_{S,a} = \inf\{ n : f_n \geq a\}$. 
\end{prop}
\begin{proof}
	Recall that the martingale $M_n$ associated to the smartingale $(f_n)$ is given by 
	 $M_n = \mathbb E_n(g f_n)$. Since $(f_n)$ has uniformly bounded differences $|\Delta f_n|\leq L$, 
	 we have by Lemma~\ref{lem:basic_estimate} with $p=1$ that the
	 martingale $(M_n)$ has uniformly bounded differences $|\Delta M_n|\leq KL$ for some 
	 constant $K$ as well. Moreover, recall
	that  $\tau_{M,a} = \inf\{ n : M_n \geq a\}$ and $\tau_{S,a} = \inf\{n : f_n \geq a\}$.
	Let $c$ be the implicit constant from Theorem~\ref{thm:pw} in inequality \eqref{eq:square_comp}.
	Let $\beta = c_3 /2$ and 
	take $\omega\in \{ S_{\tau_{S,a}}^2 \leq b, \tau_{S,a} < \infty \}$ and set $n=\tau_{S,a}(\omega)$ and 
	$m = \tau_{M,\beta a}(\omega)$. If $a>0$ is sufficiently large,
	 we have by Proposition~\ref{prop:stop_comp} that  $m\leq n$.
	Since $\omega\in \{ S_{\tau_{S,a}}^2 \leq b\}$ we have $S_n^2(\omega) \leq b$. 
	By Theorem~\ref{thm:pw}  we have $S_{M,n}^2(\omega) \leq cb$ and 
	a fortiori $S_{M,m}^2(\omega) \leq S_{M,n}^2(\omega) \leq cb$ and this gives 
	$\omega\in \{ S^2_{M,\tau_{M,\beta a}} \leq cb\}$.
	Summarizing we have 
	\[
		\{ S^2_{\tau_{S,a}} \leq b\} \subseteq \{ S_{M,\tau_{M,\beta a}}^2 \leq cb\},\qquad \{\tau_{S,a} < \infty\} \subseteq \{\tau_{M,\beta a} < \infty\}
	\]
	if $a>0$ is sufficiently large. Using Lemma~\ref{lem:mart_exp} we conclude 
	\[
		\mu(S^2_{\tau_{S,a}} \leq b, \tau_{S,a} <\infty)	\leq \mu(S_{M,\tau_{M,\beta a}}^2 \leq cb, \tau_{M,\beta a} < \infty)\leq \exp\Big(- \frac{\beta ^2 a^2}{2( \beta K L a + cb)} \Big),
	\]
	proving the desired result for some constant  $ r$.
\end{proof}

An immediate corollary of Proposition~\ref{prop:exp_ineq_smart} and 
Lemma~\ref{lem:exp_implies_lil} is the following inequality:
\begin{thm}\label{thm:lil_smartingale}
	Let $(f_n)$ be a smartingale with bounded differences $|\Delta f_n|\leq L$ such that 
	inequality \eqref{eq:diam_ineq} holds.

	Then we have, for some constant $r>0$,
	\[
		\limsup_{n\to\infty} \frac{f_n}{\sqrt{rS_n^2\log\log S_n^2}}\leq 1,\qquad \mu\text{-a.s. on } \{S^2 = \infty \}.
	\]
\end{thm}

\section{Application: Variation inequalities for smartingales}
\label{sec:appl}
In this section we apply Theorem~\ref{thm:lil_smartingale} to measures 
$\tilde{\mathbb P}$ constructed in Section~\ref{sec:tilde_measure} in order 
to deduce variation inequalities for smartingales $(f_n)$.
To this end we first prove that those measures $\tilde{\mathbb P}$ actually 
satisfy the decay inequality \eqref{eq:diam_ineq} for fat chains, which is needed 
for Theorem~\ref{thm:lil_smartingale}.

\begin{lem}\label{lem:diam_tilde}
	Assume that the siblings in $\mathscr A$ have comparable (Lebesgue) measure 
	and let $(f_n)$ be a smartingale. 

	Then, if the fat chains of atoms are constructed using
	the measure $\mu = \tilde{\mathbb P}$ for sufficiently 
	small parameter $\lambda > 0$, inequality~\eqref{eq:diam_ineq}
	is satisfied.
\end{lem}
\begin{proof}
By Theorem~\ref{thm:tilde_P} we have for each pair of atoms $B\subset A$ the inequality 
\[
\Big(\frac{\diam B}{\diam A}\Big)^{d(1+c\lambda)}\lesssim \Big(\frac{| B|}{|A|}\Big)^{1+c\lambda}	
\leq\frac{\tilde{\mathbb P}(B)}{\tilde{\mathbb P}(A)} 
\leq  \Big(\frac{|B|}{|A|}\Big)^{1-c\lambda} \lesssim \Big(\frac{\diam B}{\diam A}\Big)^{d(1-c\lambda)}
\]
for some constant $c>0$. The decomposition $(\mathscr X_j)$ of a full chain $\mathscr X$ 
into fat chains as in Lemma~\ref{lem:fat_chains} with respect to the measure $\mu = \tilde{\mathbb P}$
by construction satisfies 
\[
	\frac{\tilde{\mathbb P}(L_{j+1})}{\tilde{\mathbb P}(L_j)} \leq \frac{1}{2},
	\qquad \tilde{\mathbb P}(L_j) \leq 2 \tilde{\mathbb P}(X_j),
\]
denoting by $L_j$ and $X_j$ the largest and smallest atoms of the chain $\mathscr X_j$,
respectively.
Those inequalities give for $j\leq \ell$ 
\[
 \frac{\diam L_\ell}{\diam L_j}\lesssim 2^{(j-\ell)/(d(1+c\lambda))}, \qquad 
 \frac{\diam L_j}{\diam X_j}	\lesssim 2^{1/(d(1-c\lambda))},
\]
which in turn imply
\[
\sum_{j\leq \ell} (\diam X_j)^{-1} \lesssim \sum_{j\leq \ell} (\diam L_j)^{-1} \lesssim (\diam L_\ell)^{-1},
\]
and inequality \eqref{eq:diam_ineq} for the measure $\tilde{\mathbb P}$ is proved.
\end{proof}

Recall that $\ell$ is a non-negative integer that depends only on the space $S$ 
and the atoms $\mathscr A$. To be precise, $\ell$ is the smallest integer such that 
the assertion of Lemma~\ref{lem:existence_ell} is true and such that $2^{\ell+1} \geq \dim S + \dim S^2$.
\begin{thm}\label{thm:application}
	Let $(f_n)$ be an $\ell$-sparse smartingale with respect to Lebesgue measure 
	having uniformly bounded differences $|\Delta f_n| \leq L$.  Moreover,
	assume that the siblings in $\mathscr A$ have comparable measure.

	Then, the following assertions are true: 
	\begin{enumerate}[(i)]
		\item \label{it:infty} If $\sum_k \mathbb E_k |\Delta f_k| = \infty$, we have  
		\begin{equation*}
		\liminf_{n\to\infty} \frac{f_n}{\sum_{k\leq n} \mathbb E_k |\Delta f_k|} >0
		\end{equation*}
		on  a set of Hausdorff dimension $d$.
		\item \label{it:finite} If $(f_n)$ is uniformly bounded, we have $\sum_k \mathbb E_k |\Delta f_k| < \infty$
			on a set of Hausdorff dimension $d$.
	\end{enumerate}
\end{thm}
We remark that on every atom $B\in\mathscr A_k$ we have $\mathbb E_k |\Delta f_k| \simeq  \| \Delta f_k\|_B$.
\begin{proof}
	Let $\lambda >0$ be sufficiently small. As in equation \eqref{eq:def_f_tilde}, 
	define the sequence of functions $(\tilde{f_n})$ by its differences 
	\[
		\Delta \tilde{f_n} = \Delta f_n - \lambda p_n
	\]
	for some non-negative function $p_n\in S_n$ satisfying 
	that there exists a constant $C$ such that for all $x\in\Omega$ we have
	$C^{-1} \mathbb E_n |\Delta f_n|(x) \leq \mathbb E_n p_n(x) \leq C \mathbb E_n |\Delta f_n|(x)$.
	Let $\tilde{\mathbb P}$ be the probability measure that results from Theorem~\ref{thm:tilde_P}
	such that $(\tilde{f}_n)$ is a smartingale with respect to it.
	Denote the square functions 
	\[
		S_n^2 = \sum_{k\leq n} \mathbb E_{k}|\Delta f_k|^2,\qquad	
		\tilde{S}_n^2 = \sum_{k\leq n} \tilde{\mathbb E}_{k}|\Delta \tilde{f}_k|^2,
	\]	
	where $\mathbb E_n$ denotes the conditional expectation with respect to Lebesgue measure 
	and $\tilde{\mathbb E}_n$ denotes the conditional expectation with 
	respect to the measure $\tilde{\mathbb P}$. Note that since the siblings in $\mathscr A$ 
	have comparable measure, inequality \eqref{eq:remez} implies that those square functions
	are comparable to the smartingale square functions used on the right hand side of \eqref{eq:square_comp}.
	Moreover, 	if $\lambda>0$ is sufficiently small, the inequalities \eqref{eq:measure_ineq_dim_gtr_1}
	relating the measures $|\cdot| = \mathbb P$ and $\tilde{\mathbb P}$, together with 
	Remez' inequality \eqref{eq:remez},
	yield that $S_n^2 \simeq \tilde{S}_n^2$.

We start by proving item \eqref{it:infty} and assume that $\sum_{k} \mathbb E_k |\Delta f_k| = \infty$.

\textsc{Case I: $\tilde{\mathbb P}(S^2 = \infty) > 0$: }	

	Since inequality \eqref{eq:diam_ineq} is satisfied 
	for the measure $\tilde{\mathbb P}$ by Lemma~\ref{lem:diam_tilde}, we apply Theorem~\ref{thm:lil_smartingale}
	to the smartingale $(-\tilde{f}_n)$ with respect to the measure $\tilde{\mathbb P}$
	to deduce, for some constant $r>0$,
	\begin{equation}\label{eq:lil_appl}
		\liminf_{n\to\infty} \frac{\tilde{f}_n}{\sqrt{r \tilde{S}_n^2 \log\log \tilde{S}_n^2}}	\geq -1,\qquad 
		\text{$\tilde{\mathbb P}$-almost surely on $\{\tilde S^2 = \infty\}$}.
	\end{equation}
 Therefore, the definition of $\tilde{f}_n$ gives 
	\[
		\frac{f_n}{\sum_{k\leq n} \mathbb E_k |\Delta f_k|} \geq C^{-1}\lambda - c \frac{\sqrt{ S_n^2 \log\log S_n^2}}{\sum_{k\leq n} \mathbb E_k|\Delta f_k|},\qquad \text{as }n\to\infty,
	\]
	$\tilde{\mathbb P}$-almost surely on $\{ S^2  =\infty\}$, for some constants $c,C>0$.
	Since $|\Delta f_k| \leq L$, the rightmost term tends to zero and we obtain 
	\begin{equation}\label{eq:vargtrzero}
		\liminf_{n\to\infty }\frac{f_n}{\sum_{k\leq n} \mathbb E_k |\Delta f_k|} >0,\qquad \tilde{\mathbb P}\text{-almost surely on $\{S^2 = \infty\}$.}
	\end{equation}

	\textsc{Case II: $\tilde{\mathbb P}(S^2 = \infty) = 0$: }

	As for martingales, we also have for smartingales on $\{ S^2 < \infty\}$ that the pointwise 
	limit of $(\tilde{f}_n)$ exists and is finite $\tilde{\mathbb P}$-almost surely. Therefore, in 
	particular, the limit $\tilde{f}$ of $\tilde{f}_n = f_n - \lambda \sum_{k\leq n} p_k$ exists and is finite $\tilde{\mathbb P}$-a.s.
	on the set $\{ S^2 < \infty \}$. 
	By the pointwise comparability of $p_k$ and $\mathbb E_k|\Delta f_k|$, the finiteness of $\tilde f$,
	and $\sum_k \mathbb E_k|\Delta f_k| = \infty$, we obtain 
	\begin{equation}\label{eq:vargtrzero_square_finite}
		\lim_{n\to\infty} \frac{f_n}{\sum_{k\leq n} \mathbb E_k|\Delta f_k|} = \lambda >0,\qquad \tilde{\mathbb P}\text{-a.s. on $\{S^2 < \infty\}$}.
	\end{equation}

	Summarizing we have that $\liminf_{n\to\infty} f_n / \sum_{k\leq n} \mathbb E_k|\Delta f_k| >0$ 
	on a set of full $\tilde{\mathbb P}$ measure.	
	Recall that the atoms are convex subsets of $\mathbb R^d$. Since the siblings 
	in $\mathscr A$ have comparable measure and $|A| \simeq (\diam A)^d$ for each atom $A$,
	we can cover each ball $B$ by a uniformly finite number of disjoint atoms $A_1,\ldots, A_K$ whose diameters
	are comparable to the diameter of $B$. Then, inequality~\eqref{eq:measure_ineq_dim_gtr_1},
	applied to each atom $A_j$ as a subset of $\Omega$,  gives the inequality
	\[
		\tilde{\mathbb P}(B) \leq  \sum_{j=1}^K \tilde{\mathbb P}(A_j)	\lesssim \sum_{j=1}^K 
		|A_j|^{(1-c\lambda)}  \simeq \sum_{j=1}^K (\diam A_j)^{d(1-c\lambda)} \lesssim (\diam B)^{d(1-c\lambda)}
	\]
	for balls $B$. As a consequence, denoting $E = \{ \liminf_n f_n / \sum_{k\leq n} \mathbb E_k|\Delta f_k| >0\}$
	and by $\mathscr H^s$ the $s$-dimensional Hausdorff measure,
	\[
		\mathscr H^{d(1-c\lambda)}	(E) \gtrsim \tilde{\mathbb P}(E) >0,
	\]
	Letting $\lambda \to 0$ yields $\liminf_{n\to\infty} f_n / \sum_{k\leq n} \mathbb E_k |\Delta f_k| > 0$
	 on a set of Hausdorff dimension $d$, which proves item \eqref{it:infty}.

	In order to prove item \eqref{it:finite}, we note that assuming the uniform boundedness 
	of the smartingale $(f_n)$ and $\sum_k \mathbb E_k |\Delta f_k| = \infty$ gives
	a contradiction by the inequalities \eqref{eq:vargtrzero} and \eqref{eq:vargtrzero_square_finite}
	on a set of full $\tilde{\mathbb P}$ measure for each sufficiently small $\lambda >0$.
	As above, this gives $\sum_k |\Delta f_k| < \infty$ on a set of Hausdorff dimension $d$.
\end{proof}

\subsection*{Acknowledgments} 
The author is supported by the Austrian Science Fund FWF, project P34414.

\bibliographystyle{plain}
\bibliography{lil_splines}

\end{document}